\documentclass[reqno, twoside, a4paper]{amsart}

\usepackage{amsmath}
\usepackage{amssymb}
\usepackage{amsthm}
\usepackage{graphicx}
\usepackage{subfig}
\usepackage{enumerate}
\usepackage[all]{xy}
\usepackage{hyperref}

\hypersetup{%
pdftitle={Rational homology balls in $2$-handlebodies},%
pdfauthor={Heesang Park, Dongsoo Shin},%
pdfkeywords={rational homology ball, $2$-handlebodies},%
colorlinks=true,%
citecolor=black,%
filecolor=black,%
linkcolor=black,%
urlcolor=black,%
}

\usepackage{comment}
\usepackage{xspace}
\usepackage{mathtools}

\usepackage{tikz}

\usetikzlibrary{patterns,decorations.pathreplacing}

\tikzset{bullet/.style={
shape = circle,fill = black, inner sep = 0pt, outer sep = 0pt, minimum size = 0.35em, line width = 0pt, draw=black!100}}

\tikzset{circle/.style={
shape = circle,fill = none, inner sep = 0pt, outer sep = 0pt, minimum size = 0.35em, line width = 1pt, draw=black!100}}

\tikzset{rectangle/.style={
shape = rectangle,fill = white, inner sep = 0pt, outer sep = 0pt, minimum size = 0.35em, line width = 0pt, draw=black!100}}

\tikzset{empty/.style={
shape = circle,fill = white, inner sep = 0pt, outer sep = 0pt, minimum size = 0.35em, line width = 0pt, draw=white!100}}

\tikzset{xmark/.style={
shape = x,fill = white, inner sep = 0pt, outer sep = 0pt, minimum size = 0em, line width = 0pt, draw=white!100}}

\tikzset{longrectangle/.style={
inner sep = 1em,
rectangle,
minimum size=1em,
very thick,
draw=black!100, 
}}

\tikzset{label distance=-0.15em}

\tikzset{font=\scriptsize}

\usepackage{tikz-cd}
\usepackage{tkz-graph}

\newtheorem{theorem}{Theorem}[section]

\newtheorem{corollary}[theorem]{Corollary}

\newtheorem*{TheoremQHB}{Theorem \ref{theorem:embedded-QHB}}
\newtheorem*{TheoremQBLUP}{Theorem \ref{theorem:non-trivial-QBLUP}}

\theoremstyle{definition}

\newtheorem{definition}[theorem]{Definition}
\newtheorem{remark}[theorem]{Remark}

\numberwithin{equation}{section}

\begin{document}

\title[Rational homology balls in $2$-handlebodies]{Rational homology balls in $2$-handlebodies}

\author{Heesang Park}

\address{Department of Mathematics, Konkuk University, Seoul 05029, Korea}

\email{HeesangPark@konkuk.ac.kr}

\author{Dongsoo Shin}

\address{Department of Mathematics, Chungnam National University, Daejeon 34134, Korea}

\email{dsshin@cnu.ac.kr}



\date{June 11, 2016 / revised: February 19, 2017, June 22, 2017, July 2, 2017}

\begin{abstract}
We prove that there are rational homology balls $B_p$ smoothly embedded in the $2$-handlebodies associated to certain knots. Furthermore we show that, if we rationally blow up the $2$-handlebody along the embedded rational homology ball $B_p$, then the resulting $4$-manifold cannot be obtained just by a sequence of ordinary blow ups from the $2$-handlebody under a certain mild condition.
\end{abstract}

\maketitle


\section{Introduction}

In the theory of smooth $4$-manifolds, the rational blow-down (defined by Fintushel-Stern~\cite{Fintushel-Stern-1997}) is a useful tool for producing exotic smooth structures: We first cut out a plumbing manifold $C_p$ smoothly embedded in a given $4$-manifold $M$ and we then paste a rational homology ball $B_p$ along the boundary $\partial C_p (= \partial B_p$) to obtain a new $4$-manifold $\widetilde{M} = (M - C_p) \cup_{\partial C_p} B_p$. We recall briefly the definitions of $C_p$ and $B_p$ in Section~\ref{section:notions}. The rational blow-down surgery increases the signature while it keeps $b_2^+$ fixed. So it and its generalization (J. Park~\cite{JPark-1997}) are very useful to construct many important exotic $4$-manifolds with small Euler numbers; refer J. Park~\cite{JPark-2005}, Stipsicz-Szab\'o~\cite{Stipsicz-Szabo-2005}, J. Park-Stipsicz-Szab\'o~\cite{JPark-Stisicz-Szabo-2005}, for instance.

On the other hand, as a reverse surgery, the \emph{rational blow-up} is defined by $\overline{M}=(M-B_p) \cup_{\partial C_p} C_p$ for a $4$-manifolds $M$ where $B_p$ is smoothly embedded. The rational blow-down would also give an intriguing performance for constructing interesting $4$-manifolds. The rational blow-up will decrease the signature while it keeps $b_2^+$ fixed. Therefore it would be useful to construct minimal $4$-manifolds with $c_1^2 < 0$ for example.

In this paper we first show in Theorem~\ref{theorem:embedded-QHB} that Fintushel-Stern's rational homology ball $B_p$ is smoothly embedded into the special $2$-handlebody $M(p,m)$ associated to a certain knot $K(p,m)$, where the knot $K(p,m)$ and the 2-handlebody $M(p,m)$ are defined in Definition~\ref{definition:knot-2handlebody}. We then show in Corollary~\ref{corollary:non-trivial-QBLUP} that the rational blow-up along $B_p$ in $M(p,m)$ is not a sequence of ordinary blow-ups of $M(p,m)$.

So it would be an intriguing problem to find or construct $4$-manifolds $X$ containing $M(p,m)$ associated the knot $K(p,m)$ and to investigate what happens after rationally blowing-up $B_p$ in $X$ in case of the resulting manifold is not a sequence of blow-ups. We leave it for further studies.

\subsection{Smoothly embedded rational homology balls}

At first, the knot $K(p,m)$ and the 2-handlebody $M(p,m)$ are defined as follows:

\begin{definition}\label{definition:knot-2handlebody}
For $p \ge 2$ and $m \in \mathbb{Z}$, we denote by $K(p,m)$ the knot defined by Figure~\ref{figure:Knot}. The \emph{$2$-handlebody $M(p,m)$ associated to the knot $K(p,m)$} is defined by attaching a $2$-handle to $D^4$ along $K(p,m)$ in $\partial D^4$ with framing $p^2m-p-1$.
\end{definition}

\begin{figure}[ht]
\centering
\includegraphics[scale=0.8]{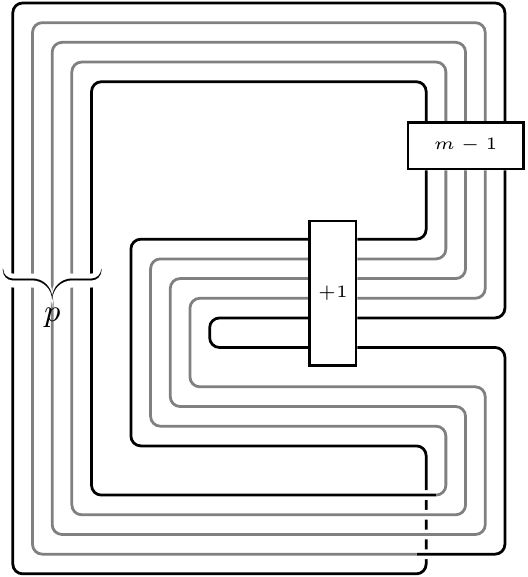}
\caption{The knot $K(p,m)$ ($p \ge 2$, $m \in \mathbb{Z}$)}
\label{figure:Knot}
\end{figure}

\begin{remark}\label{remark:referee}
The knot $K(p,m)$ is the twisted torus knot $T(p,p(m-1)+1)_{p-1,1}$ defined in Callahan-Dean-Weeks~\cite{Callahan-Dean-Weeks-1999} and Dean~\cite{Dean-1996}. According to Vafaee~\cite[Corollary~3.2]{Vafaee-2015} for example, the knot $K(p,m)$ is isotopic to a torus knot $T(p,mp-1)$. Hence the boundary of $2$-handlebody $M(p,m)$ with framing $p^2m-p-1$ is the lens space $L(p(mp-1)-1,p^2)$.
\end{remark}

One of the main theorems is the following:

\begin{theorem}\label{theorem:embedded-QHB}
For any $p \ge 2$ and $m \in \mathbb{Z}$, there is the smoothly embedded rational homology ball $B_p$ in the $2$-handlebody $M(p,m)$ associated to the knot $K(p,m)$.
\end{theorem}

We prove it in Section~\ref{section:B_p}. As a corollary, one can detect a rational homology ball $B_p$ embedded in a given $4$-manifold $X$ by looking at its Kirby diagram. Precisely:

\begin{corollary}\label{corollary:detecting-B_p}
If a Kirby diagram of $X$ contains the knot $K(p,m)$ with framing $(p^2m-p-1)$ such that no dotted circles representing a $1$-handle are linked with $K(p,m)$, then $B_p$ is embedded in $X$.
\end{corollary}

In Khodorovskiy~\cite{Khodorovskiy-2014}, she proves that the rational homology ball $B_p$ is smoothly embedded in a neighborhood of a sphere $S^2$ with self-intersection number $(-p-1)$. Theorem~\ref{theorem:embedded-QHB} may be regarded as a generalization of her results because $K(p,m)$ is an unknot for $m=0$. In PPS~\cite{PPS-2016}, the authors with J. Park generalizes Khodorovskiy~\cite{Khodorovskiy-2014} for general rational homology balls $B_{p,q}$ by using techniques developed in HTU~\cite{Hacking-Tevelev-Urzua} from the minimal model program for 3-folds in algebraic geometry. They proves the existence of rational homology balls $B_{p,q}$ smoothly embedded in the plumbing of disk bundles over spheres according to a certain linear graph which is roughly speaking half of the negative-definite plumbing graph associated to $C_{p,q}$. Recently, Owens~\cite{Owens-2017} further generalizes theorems in Khodorovskiy~\cite{Khodorovskiy-2014} and PPS~\cite{PPS-2016} and gives relatively simple topological proofs of those theorems.

But the embeddings of the rational homology balls in Khodorovskiy~\cite{Khodorovskiy-2014}, PPS~\cite{PPS-2016}, Owens~\cite{Owens-2017} are \emph{simple}; that is, the rational blow-ups of their embedded rational homology balls are just sequence of the ordinary blow-ups. It implies that their rational homology balls are not useful in general for constructing interesting 4-manifolds.

\subsection{Rational blow-up surgery}

In contrast, we show that the embedding of $B_p$ in Theorem~\ref{theorem:embedded-QHB} is \emph{not} simple under certain mild conditions in Corollary~\ref{corollary:non-trivial-QBLUP}. For this we first show:

\begin{theorem}\label{theorem:non-trivial-QBLUP}
Let $\widetilde{M}(p,m)$ be the rationally blown-up $4$-manifold along $B_p$ from $M(p,m)$. Then $\widetilde{M}(p,m)$ is the plumbing manifold with the plumbing graph in Figure~\ref{figure:After-Q-BLUP}.
\end{theorem}

\begin{figure}
\centering
\begin{tikzpicture}
\node[bullet] (10) at (1,0) [label=above:{$-(p+2)$}] {};
\node[bullet] (20) at (2,0) [label=above:{$-2$}] {};

\node[empty] (250) at (2.5,0) [] {};
\node[empty] (30) at (3,0) [] {};

\node[bullet] (350) at (3.5,0) [label=above:{$-2$}] {};
\node[bullet] (450) at (4.5,0) [label=above:{$m-1$}] {};

\draw [-] (10)--(20);
\draw [-] (20)--(250);
\draw [dotted] (20)--(350);
\draw [-] (30)--(350);
\draw [-] (350)--(450);

\draw [thick, decoration={brace,raise=0.5em}, decorate] (350) -- (20)
node [pos=0.5,anchor=south,yshift=-2em] {$p-2$};
\end{tikzpicture}
\caption{Rational blow-up of $B_p$ in $M(p,m)$}
\label{figure:After-Q-BLUP}
\end{figure}
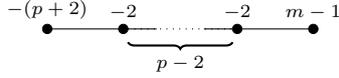

We prove it in Section~\ref{section:rational-blow-up}. We then show that the embedding of $B_p$ in Theorem~\ref{theorem:embedded-QHB} is not simple:

\begin{corollary}\label{corollary:non-trivial-QBLUP}
The rationally blown-up $4$-manifold $\widetilde{M}(p,m)$ cannot be obtained from $M(p,m)$ by any sequence of ordinary blow-ups of $M(p,m)$ either if $p$ is a positive even integer and $m$ is an odd integer or if $p \ge 2$ and $m-1 \le -2$.
\end{corollary}

\begin{proof}
If $p$ is a positive even integer and $m$ is an odd integer, then $\widetilde{M}(p,m)$ is a manifold with an even intersection form. So there are no $(-1)$-classes in the plumbing as Figure~\ref{figure:After-Q-BLUP}. On the other hand, if $m$ is sufficiently small, precisely, if $m-1 \le -2$, then there are also no $(-1)$-classes in the plumbing.
\end{proof}

\subsection{Notions}
\label{section:notions}

The plumbing diagram of $C_p$ is given as below
\begin{equation*}
\begin{tikzpicture}
\node[bullet] (10) at (1,0) [label=above:{$-(p+2)$}] {};
\node[bullet] (20) at (2,0) [label=above:{$-2$}] {};

\node[empty] (250) at (2.5,0) [] {};
\node[empty] (30) at (3,0) [] {};

\node[bullet] (350) at (3.5,0) [label=above:{$-2$}] {};

\draw [-] (10)--(20);
\draw [-] (20)--(250);
\draw [dotted] (20)--(350);
\draw [-] (30)--(350);

\draw [thick, decoration={brace,raise=0.5em}, decorate] (350) -- (20)
node [pos=0.5,anchor=south,yshift=-2em] {$p-2$};
\end{tikzpicture}
\end{equation*}
And the rational homology ball $B_p$ is given as follows: Let $F_{p-1}$ ($p \ge 2$) be the Hirzebruch surface having the negative section $s_0$ with $s_0 \cdot s_0 = -(p-1)$. Let $s_{\infty}$ be a positive section with $s_{\infty} \cdot s_{\infty} = p-1$ and $f$ a fiber. Then $B_p$ is the complement of the pair of $2$-spheres represented by the homology classes $s_{\infty} + f$ and $s_0$ in $F_{p-1}$. The Kirby diagram for $B_p$ is given in Figure~\ref{figure:Bp} (cf. Gompf-Stipsicz~\cite[Figure~8.41, p.331]{Gompf-Stipsicz-1999}).

\begin{figure}[htp]
\centering
\includegraphics{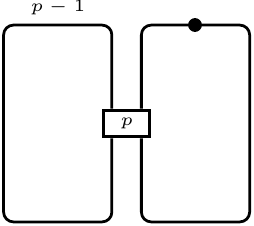}
\caption{A rational homology ball $B_p$}
\label{figure:Bp}
\end{figure}

\subsection*{Acknowledgements}

The authors would like to thank Andr\'as Stipsicz for reading the first draft of the paper and pointing out some unclear parts in definitions. We sincerely thank the referee for constructive criticisms and valuable comments (especially Remark~\ref{remark:referee} and \ref{remark:referee-2}), which were of great help in revising the manuscript. DS was supported by Chungnam National University in 2015. HP was supported by Basic Science Research Program through the National Research Foundation of Korea(NRF) funded by the Ministry of Science, ICT \& Future Planning (NRF-2015R1C1A2A01054769).

\section{Smoothly embedded rational homology balls}
\label{section:B_p}

This section is devoted to proof of Theorem~\ref{theorem:embedded-QHB}.

\begin{TheoremQHB}
For any $p \ge 2$ and $m \in \mathbb{Z}$, there is a smoothly embedded rational homology ball $B_p$ in the $2$-handlebody $M(p,m)$ associated to the knot $K(p,m)$.
\end{TheoremQHB}

\begin{figure}
\centering
\subfloat[]{\includegraphics{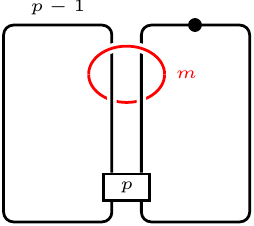}} \qquad \qquad
\subfloat[]{\includegraphics{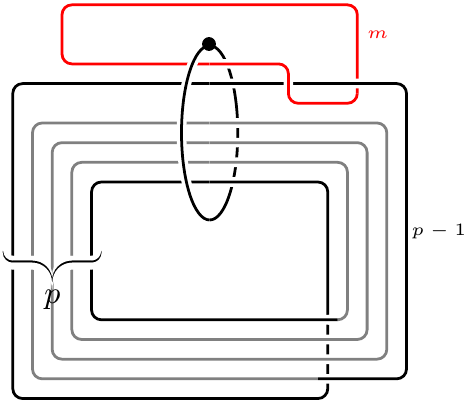}}
\caption{Proof of Theorem~\ref{theorem:embedded-QHB}}
\label{figure:Fig-C}
\end{figure}

\begin{figure}
\centering
\includegraphics{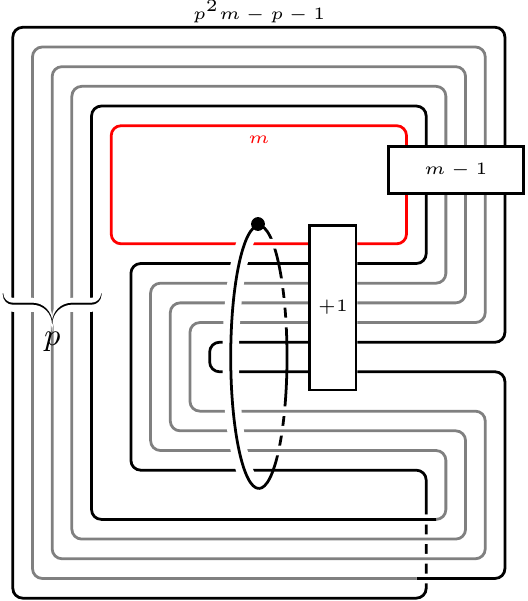}
\caption{Proof of Theorem~\ref{theorem:embedded-QHB}}
\label{figure:Fig-B}
\end{figure}

\begin{proof}
In Figure~\ref{figure:Fig-C}, the picture (B) is isotopic to (A). We apply $p$-multiple handle slide of the $(p-1)$-framed $2$-handle over the $m$-framed $2$-handle to Figure~\ref{figure:Fig-C}(B) so that we get Figure~\ref{figure:Fig-B}. Then we cancel the $1$-handle and the $m$-framed $2$-handle in Figure~\ref{figure:Fig-B}. After all, we get the $(p^2m-p-1)$-framed knot $K(p,m)$ in Figure~\ref{figure:Knot}. Note that there is a rational homology ball $B_p$ in Figure~\ref{figure:Fig-C}. Therefore $B_p$ is embedded in the $2$-handlebody associated to the knot $K(p,m)$.
\end{proof}

\section{Rational blow-ups along the rational homology balls}
\label{section:rational-blow-up}

This section is devoted to proof of Theorem~\ref{theorem:non-trivial-QBLUP}

\begin{TheoremQBLUP}
Let $\widetilde{M}(p,m)$ be the rationally blown-up $4$-manifold from $M(p,m)$. Then $\widetilde{M}(p,m)$ is the plumbing manifold with the following plumbing graph
\begin{center}
\begin{tikzpicture}
\node[bullet] (10) at (1,0) [label=above:{$-(p+2)$}] {};
\node[bullet] (20) at (2,0) [label=above:{$-2$}] {};

\node[empty] (250) at (2.5,0) [] {};
\node[empty] (30) at (3,0) [] {};

\node[bullet] (350) at (3.5,0) [label=above:{$-2$}] {};
\node[bullet] (450) at (4.5,0) [label=above:{$m-1$}] {};

\draw [-] (10)--(20);
\draw [-] (20)--(250);
\draw [dotted] (20)--(350);
\draw [-] (30)--(350);
\draw [-] (350)--(450);

\draw [thick, decoration={brace,raise=0.5em}, decorate] (350) -- (20)
node [pos=0.5,anchor=south,yshift=-2em] {$p-2$};
\end{tikzpicture}
\end{center}
\end{TheoremQBLUP}

\begin{figure}
\centering
\includegraphics{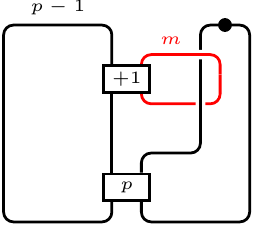}
\caption{Proof of Theorem~\ref{theorem:non-trivial-QBLUP}}
\label{figure:Fig-E}
\end{figure}

\begin{figure}
\centering
\subfloat[]{\includegraphics{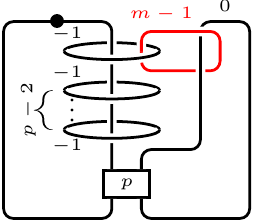}} \qquad \qquad
\subfloat[]{\includegraphics{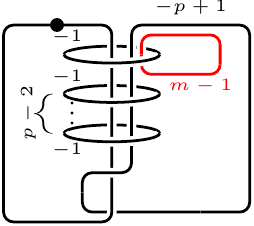}}
\caption{Proof of Theorem~\ref{theorem:non-trivial-QBLUP}}
\label{figure:Fig-FG}
\end{figure}

\begin{proof}
Note that Figure~\ref{figure:Fig-C}(A) is equal to Figure~\ref{figure:Fig-E}. We will show that Figure~\ref{figure:Fig-FG}(A) can be obtained from Figure~\ref{figure:Fig-E} by rationally blowing up $B_p$ in Figure~\ref{figure:Fig-E}.

At first, we claim that there is $C_p$ embedded in Figure~\ref{figure:Fig-FG}(A): As in Gompf-Stipsicz~\cite[12.67, 12.68, p.516]{Gompf-Stipsicz-1999}, we do a handle slide to Figure~\ref{figure:Fig-FG}(A) to get Figure~\ref{figure:Fig-FG}(B), and then, we apply another handle slide to Figure~\ref{figure:Fig-FG}(B) so that we get Figure~\ref{figure:After-Q-BLUP}, which implies that $C_p$ is embedded in Figure~\ref{figure:Fig-FG}(A).

Next, it is clear that the complement of $B_p$ in the $4$-manifold represented by Figure~\ref{figure:Fig-E} and that of $C_p$ in Figure~\ref{figure:Fig-FG}(A) are the same. Furthermore it is easy to see from the Kirby diagrams Figure~\ref{figure:Fig-E} and Figure~\ref{figure:Fig-FG}(A) that the two complements are glued to the boundary of $B_p$ in Figure~\ref{figure:Fig-E} and that of $C_p$ in Figure~\ref{figure:Fig-FG}(A), respectively, by the same gluing map. Therefore one can conclude that Figure~\ref{figure:After-Q-BLUP} is obtained from Figure~\ref{figure:Fig-C}(A) by a rational blowing-up.
\end{proof}

\begin{remark}\label{remark:referee-2}
The main part of the above proof is that the rational blow-up of Figure~\ref{figure:Fig-C}(A) is diffeomorphic to the plumbing $4$-manifold with the plumbing graph of Figure~\ref{figure:After-Q-BLUP}. The referee informed us that this result appears also in the proof of Theorem~5.1 (3) in Akbulut-Yasui~\cite{Akbulut-Yasui-2008}.
\end{remark}


\end{document}